\documentclass[12pt, reqno]{amsart}
\usepackage{amsmath, amsthm, amscd, amsfonts, amssymb, graphicx, color}
\usepackage[bookmarksnumbered, colorlinks, plainpages]{hyperref}
\usepackage{tikz}
\hypersetup{colorlinks=true,linkcolor=red, anchorcolor=green, citecolor=cyan, urlcolor=red, filecolor=magenta, pdftoolbar=true}

\newtheorem{theorem}{Theorem}[section]
\newtheorem{lemma}[theorem]{Lemma}
\newtheorem{proposition}[theorem]{Proposition}
\newtheorem{corollary}[theorem]{Corollary}
\theoremstyle{definition}

\newtheorem{example}[theorem]{Example}

\theoremstyle{remark}
\newtheorem{remark}[theorem]{\bf{Remark}}
\numberwithin{equation}{section}
\newcommand{\vertiii}[1]{{\left\vert\kern-0.25ex\left\vert\kern-0.25ex\left\vert #1 
    \right\vert\kern-0.25ex\right\vert\kern-0.25ex\right\vert}}

\allowdisplaybreaks

\begin{document}


\title[Schatten $\MakeLowercase{p}$-norm and numerical radius inequalities with applications] {{ Schatten $\MakeLowercase{p}$-norm and numerical radius inequalities with applications }}

\author[P. Bhunia and S. Sahoo] {Pintu Bhunia and Satyajit Sahoo}

\address{(Bhunia) Department of Mathematics, Indian Institute of Science, Bengaluru 560012, Karnataka, India}
\email{pintubhunia5206@gmail.com; pintubhunia@iisc.ac.in}

\address{(Sahoo) School of Mathematical Sciences, National Institute of Science Education and Research, Bhubaneswar, India}
\email{ssahoomath@gmail.com; ssahoo@niser.ac.in}

\thanks{The first author  would like to thank SERB, Govt. of India for the financial support in the form of National Post Doctoral Fellowship (N-PDF, File No. PDF/2022/000325) under the mentorship of Prof. Apoorva Khare. The second author is thankful to  NISER Bhubaneswar for providing the necessary facilities to carry out this work and also expresses his gratitude to Prof. Anil Kumar Karn.}

\subjclass[2020]{47A30, 47A12, 47A05, 05C50}
\keywords{Numerical radius, $p$-numerical radius, Schatten $p$-norm, Operator norm, Energy of graph}

\date{}
\maketitle
\begin{abstract}
We develop a new refinement of the Kato's inequality and using this refinement we obtain several upper bounds for the numerical radius of a bounded linear operator as well as the product of operators, which improve the well known existing bounds.  Further, we obtain a necessary and sufficient condition for the positivity of $2\times 2$ certain block matrices and using this condition we deduce an upper bound for the numerical radius involving a contraction operator. Furthermore, we study the Schatten $p$-norm inequalities for the sum of two $n\times n$ complex matrices via singular values and  from the inequalities we obtain the $p$-numerical radius and the classical numerical radius bounds. We show that for every $p>0$, the $p$-numerical radius $w_p(\cdot): \mathcal{M}_n(\mathbb C)\to \mathbb R$ satisfies 
$    w_p(T) \leq \frac12 \sqrt{\left\| |T|^{2(1-t)}+|T^*|^{2(1-t)} \right\|^{} \,  \big \||T|^{2t}+|T^*|^{2t}  \big\|_{p/2}^{} }  $  for all $t\in [0,1]$.
Considering $p\to \infty$, we get a nice refinement of the well known classical numerical radius bound $w(T) \leq \sqrt{\frac12 \left\| T^*T+TT^* \right \|}.$ 

\noindent As an application of the Schatten $p$-norm inequalities we develop a bound for the energy of graph. We show that 
   $  \mathcal{E}(G) \geq  \frac{2m}{   \sqrt{ \max_{1\leq i \leq n} \left\{ \sum_{j, v_i \sim v_j}d_j\right\}}  },$
where $\mathcal{E}(G)$ is the energy of a simple graph $G$ with $m$ edges and $n$ vertices $v_1,v_2,\ldots,v_n$ such that degree of $v_i$ is $d_i$ for each $i=1,2,\ldots,n.$


\end{abstract}

%

\tableofcontents

\section{Introduction and motivation}

\noindent 
Suppose $\mathcal{B}(\mathcal{H})$ denotes the $\mathbb{C}^*$-algebra of all bounded linear operators on a complex Hilbert space $\mathcal{H}.$ If $\mathcal{H}$ is finite-dimensional with dimension $n,$ then $\mathcal{B}(\mathcal{H})$ is identified with $\mathcal{M}_n(\mathbb{C})$, the set of all $n\times n$ complex matrices. For $T\in \mathcal{B}(\mathcal{H}),$ $T^*$ denotes the Hilbert adjoint of $T$ and $|T|$ denotes the positive operator $(T^*T)^{1/2}.$ Let $Re(T)=\frac12 (T+T^*)$ and $Im(T)=\frac{1}{2i}(T-T^*)$ be the real and the imaginary part of $T$, respectively. Let $\|T\|$ and $w(T)$ denote the operator norm and the numerical radius of $T$, respectively. Recall that $\|T\|=\sup_{}\{ \|Tx\|: x\in \mathcal{H}, \, \|x\|=1\} $ and $w(T)=\sup_{} \{ |\langle Tx,x\rangle|: x\in \mathcal{H}, \, \|x\|=1\}.$ 
It is easy to verify that $w(T)=\max \{ \|Re(\lambda T)\|: \lambda \in \mathbb{C},\, |\lambda|=1\},$ see \cite{Bhunia_RCMP}.
The numerical radius $w(\cdot): \mathcal{B}(\mathcal{H}) \to \mathbb {R}$ defines a norm and  it satisfies 
\begin{eqnarray}\label{01}
    \frac12 \|T\|\leq w(T) \leq \|T\| \quad \text{ for every $T\in \mathcal{B}(\mathcal{H}).$} 
\end{eqnarray}
 The numerical range and the associated numerical radius have important contribution to study operators and their various analytic and geometric properties. However, the exact value calculation of the numerical radius of an arbitrary operator is still a hard problem. So many mathematicians have been tried to develop different bounds which improve the classical bounds in \eqref{01}, see the books \cite{Book2022, Wu}. 
 One well known improved upper bound is
\begin{eqnarray}\label{k1}
    w(T) &\leq& \frac12  \big\| {|T|+|T^*|} \big\|, 
\end{eqnarray}
given by Kittaneh \cite{Kittaneh_2003}.  
 Another improvement is shown by Bhunia and Paul \cite{Bhunia_RIM_2021}, namely,
 \begin{eqnarray}\label{B2021}
    w(T) &\leq& \sqrt{ \left\| \alpha |T|^2+(1-\alpha)|T^*|^2 \right\|}, \quad \text{for all $\alpha\in [0,1]$}.
\end{eqnarray}
This bound improves the bound  $w(T) \leq \sqrt{\frac12 \left\|   |T^2|+|T^*|^2 \right\|},$ proved by Kittaneh \cite{Kittaneh_STD_2005}. 
For more recent development bounds for the numerical radius, the reader can see the articles \cite{Bhunia_ADM, BKIT, Bhunia_IJPA, Bhunia_FAA, Bhunia_BSM, Jena_Das_Sahoo_Filomat_2023, Kittaneh, Kittaneh_LAMA_2023,  Sababheh_Moradi_Sahoo_LAMA_2024, Sahoo_Rout_AIOT_2022}. 

 One of the most basic useful inequalities to study numerical radius bounds is the Cauchy-Schwarz inequality, i.e., 
$|\left<x,y\right>|\leq \|x\| \|y\|, \, \text{for all }  x,y\in\mathcal{H}.$
	An extension of this inequality is the Buzano's inequality \cite{BuzanoRSMU1974}, namely,
\begin{align}\label{Buzanoineq}
		    \left|\langle x, e\rangle \langle e, y\rangle\right|\leq \frac{\|e\|^2}{2} \left(\|x\|\|y\|+|\langle x, y\rangle|\right), \quad \text{for any $x, y, e\in\mathcal{H}$.}
		\end{align}
Closely related to the Buzano's inequality \eqref{Buzanoineq}, Dragomir \cite[Corollary 1]{Drag2017} proved that for every $T\in \mathcal B(\mathcal H)$,
\begin{equation}\label{buzano4}
\left| \left\langle Tx,Ty \right\rangle  \right|\le \frac{{{\left\| T \right\|}^{2}}}{2}\left( \left| \left\langle x,y \right\rangle  \right|+\left\| x \right\|\left\| y \right\| \right).
\end{equation}
 The Buzano's inequality and its siblings have been interestingly used in the literature to present certain applications of operator inequalities,  one can see in \cite{BuzanoRSMU1974, Drag2017, drag4, sababheh2022}.
 A refinement of the Cauchy-Schwarz inequality is given in \cite[Theorem 2.6]{Kittaneh}, namely, 
\begin{eqnarray}\label{r1}
    |\langle x,y \rangle | & \leq & \left( \|x\|- \frac{ \inf_{\lambda \in \mathbb C}\|x-\lambda y\|^2}{2\|x\|} \right) \|y\|, \quad \text{for  $x,y\in \mathcal{H}$ with $x \neq 0$.}
\end{eqnarray}
A generalization of the Cauchy-Schwarz inequality is the Kato's inequality \cite{Kato}, namely, 
\begin{eqnarray}\label{r2}
    | \langle Tx,y\rangle  | \leq \||T|^tx\|\, \||T^*|^{1-t}y\|, \quad \textit{  $T\in \mathcal{B}(\mathcal{H})$, for all $x,y\in \mathcal{H}$ and for all $ t\in[0,1]$.}
\end{eqnarray}
This is also very useful to study numerical radius bounds of bounded linear operators.

In Section \ref{sec2}, we obtain an improvement of the Kato's inequality \eqref{r2}. In particular, we prove that
\begin{theorem} (See Lemma \ref{th1})
\begin{eqnarray*}
    | \langle Tx,y\rangle  | &\leq& \begin{cases}
        \left( \left\|  |T|^tx  \right\|  - \frac{  \inf_{\lambda \in \mathbb{C}} \left\||T|^tx-\lambda |T|^{1-t}U^*y   \right\|^2} { 2 \left\| |T|^tx\right\| } \right)\left\|  |T^*|^{1-t}y  \right\| \quad \text{when $\|Tx\| \|T^*y\|\neq 0$}\\
        \left\|  |T|^tx  \right\| \left\|  |T^*|^{1-t}y  \right\| \quad  \text{when $\|Tx\| \|T^*y\|= 0$,} \quad \textit{ \textit{for all $ t\in[0,1]$}.}
    \end{cases}
        \end{eqnarray*}
\end{theorem}
Using the above improvement we develop various refinements of the existing numerical radius bounds. We show that
\begin{theorem} (See Theorem \ref{th2}) 
    \begin{eqnarray*}
        w(T) &\leq& \frac12 \left\| |T|^{2t}+|T^*|^{2(1-t)} \right\|- \lim_{n\to \infty }\frac{\left\| |T^*|^{1-t}x_n\right\|}{2\| |T|^t x_n\|} \inf_{\lambda \in \mathbb C} \left\| \left(|T|^t-\lambda|T|^{1-t}U^* \right)x_n\right \|^2,
    \end{eqnarray*} 
   \textit{for all $t\in [0,1]$} and $w(T)=\lim_{n\to \infty } | \langle Tx_n,x_n\rangle  |.$ 
\end{theorem}
This is a non-trivial refinement of $w(T)
        \leq \frac12 \left\| |T|^{2t}+|T^*|^{2(1-t)} \right\|, \, \textit{for all $t\in [0,1]$},$ proved by El-Haddad and Kittaneh \cite[Theorem 1]{El-Haddad and F. Kittaneh}.
Among other bounds we obtain that
\begin{theorem} (See Theorem \ref{111})   
\begin{eqnarray*}
w\left( T \right) &\le&  \frac{{{\left\| T \right\|}^{1-t}}}{2}\left( \left\| {{\left| T \right|}^{t}}+{{\left| {{T}^{*}} \right|}^{t}} \right\|  - \lim_{n\to \infty }\frac{\left\| |T^*|^{\frac{t}{2}}x_n\right\|}{2\left\||T|^\frac{t}{2}x_n\right\|}\inf_{\lambda\in \mathbb{C}}\left\|(|T|^{\frac{t}{2}}-\lambda |T|^{\frac{t}{2}}U^*)x_n\right\|^2\right),
\end{eqnarray*}
for all $t\in [0,1]$. 
\end{theorem}
This is a non-trivial refinement of $w(T) \le \frac{{{\left\| T \right\|}^{1-t}}}{2}  \left\| {{\left| T \right|}^{t}}+{{\left| {{T}^{*}} \right|}^{t}} \right\|, $  for all $ t\in [0,1],$ recently proved by Bhunia \cite[Corollary 2.13]{Bhunia_LAA_2024}.
Further, we develop an upper bound for the numerical radius of the product of operators. In particular, we prove that
\begin{theorem} (See Theorem \ref{1})  If $ X\geq 0$, then   
\begin{eqnarray*}
w \left( AXB \right) &\le&  \frac{\left\| X \right\|}{2}\left(\left\| {{\left| {{A}^{*}} \right|}^{2}}+{{\left| B \right|}^{2}} \right\|- \lim_{n\to \infty }\frac{\left\| A^*x_n\right\|}{2\|Bx_n\|}\inf_{\lambda\in \mathbb{C}}\left\|(B-\lambda A^*)x_n\right\|^2\right) \\
&\le&  \frac{\left\| X \right\|}{2} \left\| {{\left| {{A}^{*}} \right|}^{2}}+{{\left| B \right|}^{2}} \right\|.
\end{eqnarray*}
\end{theorem}

\noindent In Section \ref{sec2-1}, we obtain a necessary and sufficient condition for the positivity of $2\times 2$ block matrix $\left[ \begin{matrix}
   A & {{C}^{*}}  \\
   C & B  \\
\end{matrix} \right]\in \mathcal B\left( \mathcal H\oplus \mathcal H \right)$. In particular, we show that

\begin{theorem} (See Theorem \ref{lem_3.1})
    Suppose $A,B,C\in \mathcal B\left( \mathcal H \right)$ with $A\geq 0$ and $B\geq 0$. Then
  $\left[ \begin{matrix}
   A & {{C}^{*}}  \\
   C & B  \\
\end{matrix} \right]\in \mathcal B\left( \mathcal H\oplus \mathcal H \right)$ is positive  if and only if there is a contraction $K$ such that
$\left| \left\langle Cx,y \right\rangle  \right| \le \sqrt{\left\langle {{A}^{\frac{1}{2}}}{{\left| K \right|}}{{A}^{\frac{1}{2}}}x,x \right\rangle \left\langle {{B}^{\frac{1}{2}}}{{\left| {{K}^{*}} \right|}}{{B}^{\frac{1}{2}}}y,y \right\rangle }$ 
for all $x,y\in \mathcal H$. 
\end{theorem}
\noindent Applying this condition, in Theorem \ref{lem_3.2}, we obtain an upper bound for the numerical radius, which improves recently developed bound \cite[Theorem 2.5]{Sababheh_Moradi_Sahoo_LAMA_2024}.

 \noindent
 
 In Section \ref{sec3}, we study the Schatten $p$-norm and $p$-numerical radius bounds for $n\times n$ complex matrices.
For a matrix $T\in \mathcal{M}_n(\mathbb{C})$, the $p$-numerical radius of  $T$ is defined as $$w_p(T)= \max \{ \|Re(\lambda T)\|_p: \lambda\in \mathbb{C},\, |\lambda|=1\}, \quad p>0 $$ 
where $\|\cdot\|_p$ is the Schatten $p$-norm. Recall that $\|T\|_p=\left( \sum_{j=1}^n s_j^p(T)\right)^{1/p}= \left( \text{trace}\, |T|^p \right)^{1/p},$ where $s_1(T)\geq s_2(T)\geq \ldots \geq s_n(T)$ are the singular values of $T$, i.e., the eigenvalues of $|T|.$ For $p=\infty$, $\|T\|_p=\|T\|=s_1(T)$ and so $w_p(T)=w(T)$ is the classical numerical radius of $T.$
For $p\geq 1$, the Schatten $p$-norm defines a norm on $\mathcal{M}_n(\mathbb{C}).$ For $1\leq p\leq q \leq \infty,$ $\|T\|\leq \|T\|_q\leq \|T\|_p\leq \|T\|_1.$ 
The Schatten $p$-norm is very useful to study matrix theory and compact operators also. Over the years many researchers have been studied various results on it, see \cite{Natoor1, Natoor,Benm_LAMA,Bhunia_G,Conde_LAA} and the references therein. 
Here, we develop the Schatten $p$-norm inequalities for the sum of two $n\times n$ complex matrices via singular values. In particular, we show that
\begin{theorem}  (See Theorem \ref{th2-1})  
    Let $T,S \in \mathcal{M}_n(\mathbb{C})$ and $e,f,g,h$ are non-negative continuous functions on $[0,\infty)$ such that $f(t)g(t)=t$ and $e(t)h(t)=t$ for all $t\geq 0.$ Then 
    \begin{eqnarray*}
        s_j(T+S) \leq \left\| g^2(|T^*|)+h^2(|S^*|) \right\|^{1/2} \, s_j^{1/2} \left(f^2(|T|)+e^2(|S|)  \right) \quad \text{for every $j=1,2, \ldots,n$.}
    \end{eqnarray*}
    Furthermore, 
    \begin{eqnarray}\label{eq1-22}
        \|T+S\|_p \leq \left\| g^2(|T^*|)+h^2(|S^*|) \right\|^{1/2} \,  \left \|f^2(|T|)+e^2(|S|)  \right\|_{p/2}^{1/2} \quad \text{for every $p>0$ }.
    \end{eqnarray}
\end{theorem}
Among other Schatten $p$-norm inequalities, from the bound \eqref{eq1-22}, we deduce the following $p$-numerical radius bound for matrices.
\begin{theorem}   (See Corollary \ref{cor2-2}) 
    \begin{eqnarray*}
        w_p(T) &\leq& \frac12\left\| |T|^{2(1-t)}+|T^*|^{2(1-t)} \right\|^{1/2} \,  \left \||T|^{2t}+|T^*|^{2t}  \right\|_{p/2}^{1/2},
      \end{eqnarray*}  
  \text{for every $p>0$ and  for all $t\in [0,1]$ }.
Considering $p\to \infty$, we get
    \begin{eqnarray}\label{cor2-2-1-0}
        w^2(T) &\leq& \frac14\left\| |T|^{2(1-t)}+|T^*|^{2(1-t)} \right\|^{} \,  \left \||T|^{2t}+|T^*|^{2t}  \right\|_{}^{}. 
      \end{eqnarray}  
\end{theorem}
The bound \eqref{cor2-2-1-0} refines the well known bound $w^2(T) \leq \frac12 \left\| T^*T+TT^*\right\|,$ given by Kittaneh \cite{Kittaneh_STD_2005}.

Finally, in Section \ref{sec4}, as an application of the Schatten $p$-norm inequalities we develop a lower bound for the energy of a graph which is introduced by Gutman \cite{Gutman} in connection to the
total $\pi$-electron energy in Chemical Science. 
In particular, we prove that
\begin{theorem} (See Theorem \ref{Graph-lower bound})   
 \begin{eqnarray*}
     \mathcal{E}(G) &\geq&  \frac{2m}{   \sqrt{ \max_{1\leq i \leq n} \left\{ \sum_{j, v_i \sim v_j}d_j\right\}}  },  
\end{eqnarray*}
where $m$ is the number of edges of a simple graph $G$ and $d_i$ is the degree of the vertex $v_i,$ for $i=1,2,\ldots,n.$ 
\end{theorem}
This bound improves $\mathcal{E}(G)\geq 2\sqrt{m}$  (see \cite{C_lower})  for a certain class of graphs. We provide computational examples to illustrate the bounds.


\section{A refinement of the Kato's inequality and numerical radius bounds}\label{sec2}

\noindent

In this section, we obtain a refinement of the Kato's inequality \eqref{r2} and using the refinement we develop upper bounds for the numerical radius of bounded linear operators, which improve the existing bounds. To prove the results first we need the following lemma which follows from the identity $ \|x\|^2 \|y\|^2-|\langle x,y \rangle|^2= \|y\|^2  \|x-\lambda y\|^2-|\langle y,x-\lambda y\rangle|^2$ for all $\lambda \in \mathbb C$ and $x,y\in \mathcal{H}.$

\begin{lemma}\label{lem1} \cite{Kittaneh}
    If $x,y\in \mathcal{H},$ then 
    $ \|x\|^2 \|y\|^2-|\langle x,y \rangle|^2= \|y\|^2 \inf_{\lambda \in \mathbb C} \|x-\lambda y\|^2.$
\end{lemma}

 Here we mention that a simple proof of the inequality \eqref{r1} can be derived from Lemma \ref{lem1}, since
    $  \|y\|^2 \inf_{\lambda \in \mathbb C} \|x-\lambda y\|^2 = \|x\|^2 \|y\|^2-|\langle x,y \rangle|^2
      \leq 2 \|x\| \|y\| (\|x\| \|y\|-|\langle x,y \rangle|).
      $
We now obtain a refinement of the Kato's inequality \eqref{r2}, which also generalizes the refinement of the Cauchy-Schwarz inequality \eqref{r1}.


\begin{lemma}\label{th1}
    Let $T\in \mathcal{B}(\mathcal{H})$ and $T=U|T|$ be the polar decomposition of $T.$ Then for all $x,y \in \mathcal{H}$, we have 
\begin{eqnarray*}
    | \langle Tx,y\rangle  | &\leq& \begin{cases}
        \left( \left\|  |T|^tx  \right\|  - \frac{  \inf_{\lambda \in \mathbb{C}} \left\||T|^tx-\lambda |T|^{1-t}U^*y   \right\|^2} { 2 \left\| |T|^tx\right\| } \right)\left\|  |T^*|^{1-t}y  \right\| \quad \text{when $\|Tx\| \|T^*y\|\neq 0$}\\
        \left\|  |T|^tx  \right\| \left\|  |T^*|^{1-t}y  \right\| \quad  \text{when $\|Tx\| \|T^*y\|= 0$,}
    \end{cases}
        \end{eqnarray*}
  \textit{for all $ t\in[0,1]$}. In particular, for $t=\frac12$, we have
\begin{eqnarray*}
    | \langle Tx,y\rangle  | &\leq& 
    \begin{cases}
        \left( \left\|  |T|^{1/2}x  \right\| - \frac{  \inf_{\lambda \in \mathbb{C}} \left\||T|^{1/2} \left( x-\lambda U^*y  \right) \right\|^2} { 2 \left\| |T|^{1/2}x\right\| } \right)  \left\|  |T^*|^{1/2}y  \right\| \quad \text{when $\|Tx\| \|T^*y\|\neq 0$}\\
        \left\|  |T|^{1/2}x  \right\| \left\|  |T^*|^{1/2}y  \right\| \quad  \text{when $\|Tx\| \|T^*y\|= 0$.}
    \end{cases}
\end{eqnarray*}
  
\end{lemma}

\begin{proof}
If $\|Tx\| \|T^*y\|= 0$ then the result follows from \eqref{r2}, so we only prove the case when $\|Tx\| \|T^*y\|\neq 0.$
  We have
  \begin{eqnarray*}
      && \langle |T|^{2t}x,x\rangle^{1/2} \langle |T^*|^{2(1-t)}y,y\rangle^{1/2}-|\langle Tx,y\rangle|\\
      &=&  \frac{\langle |T|^{2t}x,x\rangle^{} \langle |T^*|^{2(1-t)}y,y\rangle^{}-|\langle Tx,y\rangle|^2}{\langle |T|^{2t}x,x\rangle^{1/2} \langle |T^*|^{2(1-t)}y,y\rangle^{1/2}+|\langle Tx,y\rangle| }\\
      &\geq&  \frac{\| |T|^{t}x\|^2 \| |T^*|^{1-t}y\|^2-|\langle  |T|^t x,|T|^{1-t}U^*y\rangle|^2}{2 \langle |T|^{2t}x,x\rangle^{1/2} \langle |T^*|^{2(1-t)}y,y\rangle^{1/2} } \quad (\text{by using \eqref{r2}})\\
      &=&  \frac{\| |T|^{t}x\|^2 \| |T|^{1-t}U^*y\|^2-|\langle  |T|^t x,|T|^{1-t}U^*y\rangle|^2}{2 \langle |T|^{2t}x,x\rangle^{1/2} \langle |T^*|^{2(1-t)}y,y\rangle^{1/2} }\\
&=& \frac{\left\|  |T^*|^{1-t}y  \right\|^2  \inf_{\lambda \in \mathbb{C}} \left\||T|^tx-\lambda |T|^{1-t}U^*y   \right\|^2} { 2 \left\| |T|^tx\right\| \left\|  |T^*|^{1-t}y  \right\|} \quad (\textit{by Lemma \ref{lem1}}).
\end{eqnarray*}
This implies
\begin{eqnarray*}
    | \langle Tx,y\rangle  | \leq \langle |T|^{2t}x,x\rangle^{1/2} \langle |T^*|^{2(1-t)}y,y\rangle^{1/2}- \frac{\left\|  |T^*|^{1-t}y  \right\|^2  \inf_{\lambda \in \mathbb{C}} \left\||T|^tx-\lambda |T|^{1-t}U^*y   \right\|^2} { 2 \left\| |T|^tx\right\| \left\|  |T^*|^{1-t}y  \right\|},
\end{eqnarray*}
as desired.
\end{proof}

From Lemma \ref{th1}, we obtain the following inequality for positive operators.

\begin{proposition} \label{cor1}
    If $T\in \mathcal{B}(\mathcal{H})$ is positive then for all $x,y \in \mathcal{H}$, we have
    \begin{eqnarray*}
    | \langle Tx,y\rangle  | &\leq& \begin{cases}
        \left(  \left\|  T^tx  \right\|  - \frac{  \inf_{\lambda \in \mathbb{C}} \left\|T^tx-\lambda T^{1-t}y   \right\|^2} { 2 \left\| T^tx\right\| } \right) \left\|  T^{1-t}y  \right\| \quad  \text{when $\|Tx\| \|Ty\|\neq 0$} \\
        \left\|  T^tx  \right\| \left\|  T^{1-t}y  \right\| \quad  \text{when $\|Tx\| \|Ty\|= 0$,}
    \end{cases}
\end{eqnarray*}
 \textit{for all $ t\in[0,1]$}. In particular, for $t=\frac12$, we have
\begin{eqnarray*}
    | \langle Tx,y\rangle  | &\leq& \begin{cases}
        \left( \left\|  T^{1/2}x  \right\|  - \frac{  \inf_{\lambda \in \mathbb{C}} \left\| T^{1/2} \left( x-\lambda y  \right) \right\|^2} { 2 \left\| T^{1/2}x\right\| } \right ) \left\|  T^{1/2}y  \right\| \quad  \text{when $\|Tx\| \|Ty\|\neq 0$} \\
        \left\|  T^{1/2}x  \right\| \left\|  T^{1/2}y  \right\| \quad  \text{when $\|Tx\| \|Ty\|= 0$.}
        
    \end{cases} 
\end{eqnarray*}
  
\end{proposition}

Here, we consider an example to show that Lemma \ref{th1} is a proper refinement of \eqref{r2}.

\begin{example}
    Take $\mathcal{H}={\mathbb C}^n$, $T=I,$ $x=(1,0,0,\ldots,0)$ and $y=(\frac{1}{\sqrt{2}}, \frac{1}{\sqrt{2}},0,\ldots,0).$
We have
\begin{eqnarray*}
 | \langle Tx,y\rangle  |=\frac{1}{\sqrt{2}} 
 &<& \frac{3}{4}=\left( \left\|  |T|^tx  \right\|  - \frac{  \inf_{\lambda \in \mathbb{C}} \left\||T|^tx-\lambda |T|^{1-t}U^*y   \right\|^2} { 2 \left\| |T|^tx\right\| } \right)\left\|  |T^*|^{1-t}y  \right\| \\
 &<& 1=\||T|^tx\| \||T^*|^{1-t}y\|.
\end{eqnarray*}
\end{example}

Applying Lemma \ref{th1}, we now obtain a necessary and sufficient condition for which the Kato's inequality \eqref{r2} becomes equality. 
Note that when $\|Tx\| \|T^*y\|= 0$ then
$ | \langle Tx,y\rangle  | = \||T|^tx\| \||T^*|^{1-t}y\|=0$  \textit{ for every $t\in [0,1]$ }.

\begin{proposition}
     Let $T\in \mathcal{B}(\mathcal{H})$ and let $x,y \in \mathcal{H}$ be such that $\|Tx\| \|T^*y\|\neq 0.$ Then
$$ | \langle Tx,y\rangle  | = \||T|^tx\| \||T^*|^{1-t}y\| \quad \textit{ for every $t\in [0,1]$ }$$
if and only if 
$$\inf_{\lambda \in \mathbb{C}} \left\||T|^tx-\lambda |T|^{1-t}U^*y   \right\|=0\quad \textit{$(T=U|T|$ is the polar decomposition of $T).$}$$
\end{proposition}

\begin{proof}
    The necessary part follows from Lemma \ref{th1} and so we only prove the sufficient part. Suppose $\inf_{\lambda \in \mathbb{C}} \left\||T|^tx-\lambda |T|^{1-t}U^*y   \right\|=0.$ Therefore, this implies $|\langle |T|^tx, |T|^{1-t}U^*y\rangle|=\||T|^tx\| \||T|^{1-t}U^*y\|,$ i.e., $ |\langle Tx,y\rangle|  = \||T|^tx\| \||T^*|^{1-t}y\|.$
\end{proof}

Again applying Lemma \ref{th1}, we develop an upper bound for the numerical radius of bounded linear operators which refines the existing bound \eqref{k1}.

\begin{theorem}\label{th2}
    Let $T\in \mathcal{B}(\mathcal{H})$, $T\neq 0$ and $T=U|T|$ be the polar decomposition. Then 
    \begin{eqnarray*} 
        w(T) &\leq& \frac12 \left\| |T|^{2t}+|T^*|^{2(1-t)} \right\|- \lim_{n\to \infty }\frac{\left\| |T^*|^{1-t}x_n\right\|}{2\| |T|^t x_n\|} \inf_{\lambda \in \mathbb C} \left\| \left(|T|^t-\lambda|T|^{1-t}U^* \right)x_n\right \|^2 \\
        &\leq& \frac12 \left\| |T|^{2t}+|T^*|^{2(1-t)} \right\|, \quad \textit{for all $t\in [0,1]$}\notag
    \end{eqnarray*} 
    where the sequence $\{ x_n\}\subset \mathcal{H}, \|x_n\|=1$ such that $w(T)=\lim_{n\to \infty } | \langle Tx_n,x_n\rangle  |.$
   For $t=\frac12$, 
     \begin{eqnarray}\label{n2}
        w(T) &\leq& \frac12 \left\| |T|^{}+|T^*|^{} \right\|-\lim_{n\to \infty }\frac{\left\||T^*|^{1/2}x_n\right\| }{2 \left\||T|^{1/2}x_n\right\|^{}} \inf_{\lambda \in \mathbb C} \left\||T|^{1/2} \left( I -\lambda U^* \right)x_n \right\|^2\\
        &\leq& \frac12 \left\| |T|^{}+|T^*|^{} \right\|. \notag
    \end{eqnarray}
    
\end{theorem}
\begin{proof}
    Put $x=x_n $ and $ y=x_n$ in Lemma \ref{th1} and taking $\lim_{n\to \infty}$, we get the desired bounds. 
    \end{proof}

       Clearly, Theorem \ref{th2} refines the existing bound $w(T) \leq \frac12 \left\| |T|^{2t}+|T^*|^{2(1-t)} \right\|, \quad \textit{for all } t\in [0,1], $
which is proved by El-Haddad and Kittaneh \cite[Theorem 1]{El-Haddad and F. Kittaneh}. We show in Example \ref{exampleR} that this refinement is proper.
Next, we obtain an inner product inequality for the product of operators.
 
\begin{lemma}\label{Lemma-product}
Let $A, X, B\in \mathcal B(\mathcal H)$ with $X$ is positive. Then	$$\left| \left\langle AXBx,y \right\rangle  \right|\le \frac{\left\| X \right\|}{2}\left(2\|Bx\|\|A^*y\|-\frac{\inf_{\lambda\in \mathbb{C}}\|Bx-\lambda A^*y\|^2 }{2\|Bx\|}\|A^*y\|\right), $$
\textit{for all $x,y\in \mathcal H$ with $Bx\neq 0$.}
\end{lemma}
\begin{proof}
\textit{For any $T\in \mathcal B(\mathcal H)$},
we can write \eqref{buzano4} as 
	\[\begin{aligned}
   \left| \left\langle {{\left| T \right|}^{2}}x,y \right\rangle  \right| 
   &\le\frac{{{\left\| \;\left| T \right| \;\right\|}^{2}}}{2}\left( \left| \left\langle x,y \right\rangle  \right|+\left\| x \right\|\left\| y \right\| \right)  
  =\frac{\left\| {{\left| T \right|}^{2}} \right\|}{2}\left( \left| \left\langle x,y \right\rangle  \right|+\left\| x \right\|\left\| y \right\| \right). 
\end{aligned}\]
Substitute ${{\left| T \right|}^{2}}$ by $X$, we get 
\begin{equation}\label{5}
\left| \left\langle Xx,y \right\rangle  \right|\le \frac{\left\| X \right\|}{2}\left( \left| \left\langle x,y \right\rangle  \right|+\left\| x \right\|\left\| y \right\| \right)
\end{equation}
(see also \cite[Remark 3.1]{sababheh2022}). Now we replace $x$ and $y$ by $Bx$ and ${{A}^{*}}y$ respectively, we get
	\[\left| \left\langle AXBx,y \right\rangle  \right|\le \frac{\left\| X \right\|}{2}\left( \left| \left\langle Bx,A^*y \right\rangle  \right|+\left\| Bx \right\|\left\| {{A}^{*}}y \right\| \right).\]
 Therefore, using \eqref{r1}, we get
 \begin{align*}
     \left| \left\langle AXBx,y \right\rangle  \right|&\le \frac{\left\| X \right\|}{2}\left[\left(\|Bx\|-\frac{\inf_{\lambda\in \mathbb{C}}\|Bx-\lambda A^*y\|^2 }{2\|Bx\|}\right)\|A^*y\|+\left\| Bx \right\|\left\| {{A}^{*}}y \right\| \right]\\
     &=\frac{\left\| X \right\|}{2}\left(2\|Bx\|\|A^*y\|-\frac{\inf_{\lambda\in \mathbb{C}}\|Bx-\lambda A^*y\|^2 }{2\|Bx\|}\|A^*y\|\right).
 \end{align*}
\end{proof}

If we consider $A=U{{\left| T \right|}^{\frac{1-t}{2}}}$ (where $T=U\left| T \right|$ is the polar decomposition), $B={{\left| T \right|}^{\frac{t}{2}}}$ and $X={{\left| T \right|}^{\frac{1}{2}}}$ in Lemma \ref{Lemma-product}, we obtain

\begin{proposition}\label{7}
Let $T \in \mathcal B(\mathcal H)$ and  $T=U\left| T \right|$ be the polar decomposition of $T$. Then
\begin{equation*}\label{20}
\left| \left\langle Tx,y \right\rangle  \right|\le \frac{{{\left\| T \right\|}^{\frac{1}{2}}}}{2}\left(2\sqrt{\left\langle {{\left| T \right|}^{t}}x,x \right\rangle \left\langle {{\left| T^* \right|}^{1-t}}y,y \right\rangle }-\frac{\inf_{\lambda\in \mathbb{C}}\||T|^{\frac{t}{2}}x-\lambda |T|^{\frac{1-t}{2}}U^*y\|^2 }{2\||T|^{\frac{t}{2}}x\|}\||T^*|^{\frac{1-t}{2}}y\| \right),
\end{equation*}
for all $t\in [0,1]$ and for any $x,y\in \mathcal H$ with $Tx\neq 0$.
\end{proposition}


Again, by setting $A=U{{\left| T \right|}^{\frac{t}{2}}}$, $B={{\left| T \right|}^{\frac{t}{2}}}$  and $X={{\left| T \right|}^{1-t}}$ in Lemma \ref{Lemma-product}, we get

\begin{proposition}\label{Thm_3prod}
Let $T \in \mathcal B(\mathcal H)$ and $T=U\left| T \right|$ be the polar decomposition.  Then
\begin{equation*} 
\left| \left\langle Tx,y \right\rangle  \right|\le \frac{{{\left\| T \right\|}^{1-t}}}{2}\left(2\sqrt{\left\langle {{\left| T \right|}^{t}}x,x \right\rangle \left\langle {{\left| T^* \right|}^{t}}y,y \right\rangle }-\frac{\inf_{\lambda\in \mathbb{C}}\||T|^{\frac{t}{2}}x-\lambda |T|^{\frac{t}{2}}U^*y\|^2 }{2\||T|^{\frac{t}{2}}x\|}\||T^*|^{\frac{t}{2}}y\| \right),
\end{equation*}
for all  $t\in [0,1]$ and for any $x,y\in \mathcal H$ with $Tx\neq 0$.
\end{proposition}

Consider $x=x_n$ and $ y=x_n$ in Proposition \ref{Thm_3prod} and taking  $\lim_{n\to \infty}$, we get the following stronger upper bound for the numerical radius than that in \eqref{01}. 

\begin{theorem}\label{Thm_refinement}
    Let $T \in \mathcal B(\mathcal H)$, $T\neq 0$ and $T=U\left| T \right|$ be the polar decomposition. Then 
\begin{eqnarray*}\label{Ineq_3prod_0}
w(T)&\leq &\|T\|-{\left\| T \right\|}^{1-t} \lim_{n\to \infty }\frac{\left\||T^*|^{\frac{t}{2}}x_n\right\|}{4\||T|^\frac{t}{2}x_n\|}\inf_{\lambda\in \mathbb{C}}\left\|(|T|^{\frac{t}{2}}-\lambda |T|^{\frac{t}{2}}U^*)x_n\right\|^2, \quad \textit{for all $t\in [0,1]$ }\\
& \leq & \|T\|,
\end{eqnarray*}
where the sequence $\{ x_n\}\subset \mathcal{H}, \|x_n\|=1$ such that $w(T)=\lim_{n\to \infty } | \langle Tx_n,x_n\rangle  |.$
\end{theorem}


Again, from Proposition \ref{7} we obtain 

\begin{theorem}\label{10}
Let $ T \in \mathcal B(\mathcal H)$, $T\neq 0$ and  $T=U\left| T \right|$ be the polar decomposition. Then 
\begin{eqnarray*} 
w \left( T \right) &\le&  \frac{{{\left\| T \right\|}^{\frac{1}{2}}}}{2}\left( \left\| {{\left| T \right|}^{t}}+{{\left| {{T}^{*}} \right|}^{1-t}} \right\|  - \lim_{n\to \infty }\frac{\left\| |T^*|^{\frac{1-t}{2}}x_n\right\|}{2\||T|^\frac{t}{2}x_n\|}\inf_{\lambda\in \mathbb{C}}\left\|(|T|^{\frac{t}{2}}-\lambda |T|^{\frac{1-t}{2}}U^*)x_n\right\|^2\right)\\
&\le& \frac{{{\left\| T \right\|}^{\frac{1}{2}}}}{2} \left\| {{\left| T \right|}^{t}}+{{\left| {{T}^{*}} \right|}^{1-t}} \right\|, \quad \textit{for all $t\in [0,1]$, }
\end{eqnarray*}
where the sequence $\{ x_n\}\subset \mathcal{H},  \|x_n\|=1$ such that $w(T)=\lim_{n\to \infty } | \langle Tx_n,x_n\rangle  |.$
\end{theorem}

\begin{proof}
From Proposition \ref{7}, we get
\[\begin{aligned}
   \left| \left\langle Tx,x \right\rangle  \right|& \le \frac{{{\left\| T \right\|}^{\frac{1}{2}}}}{2}\left(\left\langle {{\left| T \right|}^{t}}x,x \right\rangle +\left\langle {{\left| T^* \right|}^{1-t}}x,x \right\rangle -\frac{\inf_{\lambda\in \mathbb{C}}\||T|^{\frac{t}{2}}x-\lambda |T|^{\frac{1-t}{2}}U^*x\|^2 }{2\||T|^{\frac{t}{2}}x\|}\||T^*|^{\frac{1-t}{2}}x\| \right) \\ 
 & \le \frac{{{\left\| T \right\|}^{\frac{1}{2}}}}{2}\left( \left\| {{\left| T \right|}^{t}}+{{\left| {{T}^{*}} \right|}^{1-t}} \right\|  -\frac{\inf_{\lambda\in \mathbb{C}}\||T|^{\frac{t}{2}}x-\lambda |T|^{\frac{1-t}{2}}U^*x\|^2 }{2\||T|^{\frac{t}{2}}x\|}\||T^*|^{\frac{1-t}{2}}x\|\right), \quad  
\end{aligned}\]
where $Tx\neq 0.$ This gives the desired bounds.
\end{proof}

Similarly, form Proposition \ref{Thm_3prod}, we obtain
\begin{theorem}\label{111}
    Let $T \in \mathcal B(\mathcal H)$, $T\neq 0$ and $T=U\left| T \right|$ be the polar decomposition. Then 
\begin{eqnarray*}
w\left( T \right) &\le&  \frac{{{\left\| T \right\|}^{1-t}}}{2}\left( \left\| {{\left| T \right|}^{t}}+{{\left| {{T}^{*}} \right|}^{t}} \right\|  - \lim_{n\to \infty }\frac{\left\| |T^*|^{\frac{t}{2}}x_n\right\|}{2\left\||T|^\frac{t}{2}x_n\right\|}\inf_{\lambda\in \mathbb{C}}\left\|(|T|^{\frac{t}{2}}-\lambda |T|^{\frac{t}{2}}U^*)x_n\right\|^2\right)\\
&\le& \frac{{{\left\| T \right\|}^{1-t}}}{2}  \left\| {{\left| T \right|}^{t}}+{{\left| {{T}^{*}} \right|}^{t}} \right\|, \quad \mbox{for all}~  t\in [0,1],
\end{eqnarray*}
 where the sequence $\{ x_n\}\subset \mathcal{H},  \|x_n\|=1$ such that $w(T)=\lim_{n\to \infty } | \langle Tx_n,x_n\rangle  |.$
\end{theorem}

\begin{remark}
    Recently, Bhunia \cite[Corollary 2.13]{Bhunia_LAA_2024} proved that
    \begin{eqnarray}\label{EQ1}
        w(T) &\leq& \frac{\|T\|^t}{2}  \left\| |T|^{2\alpha(1-t)}+ |T^*|^{2(1-\alpha)(1-t)}  \right\|, \quad 0\leq \alpha,t\leq 1.
    \end{eqnarray}
    In particular, for $\alpha=\frac12$,
    \begin{eqnarray}\label{EQ2}
        w(T) &\leq& \frac{\|T\|^t}{2}\left\| |T|^{1-t}+ |T^*|^{1-t}  \right\|, \quad 0\leq t\leq 1
    \end{eqnarray}
    and also for $t=\frac12 ,$
    \begin{eqnarray}\label{EQ3}
        w(T) &\leq& \frac{\|T\|^{1/2}}{2}  \left\| |T|^{\alpha}+ |T^*|^{1-\alpha}  \right\|, \quad 0\leq \alpha \leq 1.
    \end{eqnarray}
    Again if we take $\alpha=\frac12$ in \eqref{EQ3}, then
     $w(T) \leq \frac12 \|T\|^{1/2} \left\| |T|^{1/2}+ |T^*|^{1/2}  \right\|,$
    which is also proved by Kittaneh et al. \cite{Kittaneh_LAMA_2023}. Here we would like to remark that Theorem \ref{10} refines the bound \eqref{EQ3} and  Theorem \ref{111} refines the bound \eqref{EQ2}. To show proper refinement we consider the following example.
\end{remark}
\begin{example}\label{exampleR}
   Take $T=\begin{bmatrix}
        0 & 1 & 0\\
       0 &0 &2\\
        0 &0&0
    \end{bmatrix}.$ Then  $U=\begin{bmatrix}
        0 & 1 & 0\\
       0 &0 &1\\
        0 &0&0
    \end{bmatrix}$ such that $T=U|T|$ is the polar decomposition. Here,  $\|T\|=2$, $w(T)=\frac{\sqrt{5}}{2}=|\langle Tx_0, x_0 \rangle|,$ where $x_0=\left(\frac{1}{\sqrt{10}}, \frac{1}{\sqrt{2}},\frac{\sqrt{2}}{\sqrt{5}}\right)\in \mathbb {C}^3$. Also, simple calculation shows that
        $\inf_{\lambda\in \mathbb{C}} \left\|(|T|^{\frac{t}{2}}-\lambda |T|^{\frac{t}{2}}U^*)x_0\right\|^2= $  $ \inf_{\lambda \in \mathbb{C}} \left\{  \frac{1}{2}\left|1-\frac{\lambda}{\sqrt{5}}\right|^2+2^t\left|\frac{\sqrt{2}}{\sqrt{5}}-\frac{\lambda}{\sqrt{2}}\right|^2\right\}> 0$, $\left\||T^*|^{\frac{t}{2}}x_0\right\|=\sqrt{\frac{1}{10}+\frac{2^t}{2}}>0$ and $\||T|^\frac{t}{2}x_0\|=\sqrt{\frac{1}{2}+\frac{2^{1+t}}{5}}>0$.
         In particular, \text{for $t=0.01$}, we see that  $${\left\| T \right\|}^{1-t} \frac{\left\||T^*|^{\frac{t}{2}}x_0\right\|}{4\||T|^\frac{t}{2}x_0\|}\inf_{\lambda\in \mathbb{C}}\left\|(|T|^{\frac{t}{2}}-\lambda |T|^{\frac{t}{2}}U^*)x_0\right\|^2\approx 0.06089 $$ and so
          $ \|T\|-{\left\| T \right\|}^{1-t} \frac{\left\||T^*|^{\frac{t}{2}}x_0\right\|}{4\||T|^\frac{t}{2}x_0\|}\inf_{\lambda\in \mathbb{C}}\left\|(|T|^{\frac{t}{2}}-\lambda |T|^{\frac{t}{2}}U^*)x_0\right\|^2\approx 1.9391 < 2=\|T\|$
         and so the refinement in Theorem \ref{Thm_refinement} is a proper.
         Similarly, using the same example we can show that the refinements in Theorem \ref{th2}, Theorem \ref{10} and Theorem \ref{111} are also proper.
\end{example}

 Next, by using Lemma \ref{Lemma-product}, we obtain the following numerical radius bound for the product of operators. 

\begin{theorem}\label{1}
Let $A, B, X \in \mathcal B(\mathcal H)$ with $B\neq 0$ and $X$ is positive. Then
\begin{eqnarray*}\label{buzano2}
w \left( AXB \right) &\le&  \frac{\left\| X \right\|}{2}\left(\left\| {{\left| {{A}^{*}} \right|}^{2}}+{{\left| B \right|}^{2}} \right\|- \lim_{n\to \infty }\frac{\left\| A^*x_n\right\|}{2\|Bx_n\|}\inf_{\lambda\in \mathbb{C}}\left\|(B-\lambda A^*)x_n\right\|^2\right)\\
&\le&  \frac{\left\| X \right\|}{2} \left\| {{\left| {{A}^{*}} \right|}^{2}}+{{\left| B \right|}^{2}} \right\|,
\end{eqnarray*}
where the sequence $\{ x_n\}\subset \mathcal{H},  \|x_n\|=1$ such that $w(AXB)=\lim_{n\to \infty } | \langle AXB x_n,x_n\rangle  |.$\\
For $X=I$,
\begin{eqnarray}\label{Ineq_Particular}
w \left( AB \right) &\le&   \frac{1}{2}\left\|{{\left| {{A}^{*}} \right|}^{2}}+{{\left| B \right|}^{2}} \right\|-\lim_{n\to \infty }\frac{\left\| A^*x_n\right\|}{4\|Bx_n\|}\inf_{\lambda\in \mathbb{C}}\left\|(B-\lambda A^*)x_n\right\|^2\\
&\le&  \frac{1}{2} \left\|{{\left| {{A}^{*}} \right|}^{2}}+{{\left| B \right|}^{2}} \right\| \notag.
\end{eqnarray}

\end{theorem}

\begin{proof}
From Lemma \ref{Lemma-product}, we obtain
\[\begin{aligned}
   \left| \left\langle AXBx,x \right\rangle  \right|&\le \frac{\left\| X \right\|}{2}\left(2\|Bx\|\|A^*x\|-\frac{\inf_{\lambda\in \mathbb{C}}\|Bx-\lambda A^*x\|^2 }{2\|Bx\|}\|A^*x\|\right) \\ 
 & \le \frac{\left\| X \right\|}{2}\left(\left\langle \left( {{\left| {{A}^{*}} \right|}^{2}}+{{\left| B \right|}^{2}} \right)x,x \right\rangle -\frac{\inf_{\lambda\in \mathbb{C}}\|Bx-\lambda A^*x\|^2 }{2\|Bx\|}\|A^*x\|\right) \\  
 & \le \frac{\left\| X \right\|}{2}\left(\left\| {{\left| {{A}^{*}} \right|}^{2}}+{{\left| B \right|}^{2}} \right\|-\frac{\inf_{\lambda\in \mathbb{C}}\|Bx-\lambda A^*x\|^2 }{2\|Bx\|}\|A^*x\|\right), 
\end{aligned}\]
where $x\in \mathcal H$ with $\|x\|=1$ and $Bx\neq 0.$ This gives the desired bounds.

\end{proof}


\begin{remark}
    (i) Let $A, X \in \mathcal B(\mathcal H)$ with $A\neq 0$ and $X$ is positive.
    Put $B=A$ in Theorem \ref{1}, we get $w(AXA) \leq \frac{\|X\|}{2} \| A^*A+AA^*\|.$ \\
(ii) The inequality \eqref{Ineq_Particular} is a non-trivial improvement of 
    $
    w \left( AB \right)     \le  \frac{1}{2} \left\|{{\left| {{A}^{*}} \right|}^{2}}+{{\left| B \right|}^{2}} \right\|,
        $
        which follows from \cite[Remark 1, (17)]{Kittaneh_STD_2005}.
     This also follows from \cite[Proposition 2.8]{Satt} by setting $f(t)=g(t)=t^\frac{1}{2}, r=1, p=q=2, X=I$. \\
 (iii) By setting $\alpha=1, r=1, p=q=2$ in \cite[Theorem 3.1]{Satt} we have that
    $
    w \left( AXB \right)     \le  \frac{\|X\|}{2} \left\|{{ {{A}} }^{2}}+{{ B }^{2}} \right\|,
        $
       for $A, B, X \in \mathcal{B(H)}$ such that $A, B$ are positive.   In particular,
        $
            w \left( AB \right)     \le  \frac{1}{2} \left\|{{ {{A}} }^{2}}+{{ B }^{2}} \right\|.
        $
       Note that the bounds in Theorem \ref{1} are better (more generalize) than these above bounds.
\end{remark}

Suppose $S=U\left| S \right|$ is the polar decomposition. Letting $A=U{{\left| S \right|}^{\frac{1}{2}}}$, $X={{\left| S \right|}^{\frac{1}{2}}}$ and $B=T$ in Theorem \ref{1}, we obtain

\begin{corollary}\label{Cor_product_00}
   Let $T, S \in \mathcal B(\mathcal H)$, $T\neq 0$, $S=U\left| S \right|$ be the polar decomposition. Then
\begin{equation*}
w \left( ST \right)\le   \frac{\|S\|^{\frac{1}{2}}}{2}\left(\left\|{{\left| {{S}^{*}} \right|}}+{{\left| T \right|}^{2}} \right\|- \lim_{n\to \infty}\frac{\left\| |S^*|^{\frac{1}{2}}x_n\right\|}{2\|Tx_n\|}\inf_{\lambda\in \mathbb{C}}\|(T-\lambda|S^*|^{\frac{1}{2}}U^*)x_n\|^2\right),
\end{equation*}
where the sequence $\{ x_n\}\subset \mathcal{H},  \|x_n\|=1$ such that $w(ST)=\lim_{n\to \infty } | \langle STx_n,x_n\rangle  |.$
\end{corollary}

We now provide an example to illustrate the numerical radius bounds in Theorem \ref{1}. 

\begin{example}
   Take $A=\begin{bmatrix}
        0 & 1 & 0\\
       0 &0 &2\\
        0 &0&0
    \end{bmatrix}$ and $X=B=I$. We have  $\|A\|=2$ and $w(A)=\frac{\sqrt{5}}{2}=|\langle Ax_0, x_0 \rangle|,$ where $x_0=\left(\frac{1}{\sqrt{10}}, \frac{1}{\sqrt{2}},\frac{\sqrt{2}}{\sqrt{5}}\right)\in \mathbb {C}^3$. Also, $\inf_{\lambda\in \mathbb{C}}\left\|(B-\lambda A^*)x_0\right\|^2 =\inf_{\lambda\in \mathbb{C}} \left\{\frac{1}{10}+\left|\frac{1}{\sqrt{2}}-\frac{\lambda}{\sqrt{10}}\right|^2+\left|\frac{\sqrt{2}}{\sqrt{5} }-\lambda \sqrt{2}\right|^2\right\}\approx 0.405$, $\|A^*x_0\|=\sqrt{\frac{21}{10}}$ and $\|Bx_0\|=1.$ Therefore, 
$\frac{\left\| A^*x_0\right\|}{2\|Bx_0\|}\inf_{\lambda\in \mathbb{C}}\left\|(B-\lambda A^*)x_0\right\|^2\approx 0.2934 $
and so
$$ \frac{\left\| X \right\|}{2} \left( \left\| {{\left| {{A}^{*}} \right|}^{2}}+{{\left| B \right|}^{2}} \right\|- \frac{\left\| A^*x_0\right\|}{2\|Bx_0\|}\inf_{\lambda\in \mathbb{C}}\left\|(B-\lambda A^*)x_0\right\|^2 \right)\approx 2.3532< 2.5=\frac{\left\| X \right\|}{2} \left\| {{\left| {{A}^{*}} \right|}^{2}}+{{\left| B \right|}^{2}} \right\|.$$
Thus, Theorem \ref{1} is a proper refinement of $w(AXB) \leq \frac{\left\| X \right\|}{2} \left\| {{\left| {{A}^{*}} \right|}^{2}}+{{\left| B \right|}^{2}} \right\|.$
\end{example}

\section{A Refinement of the numerical radius bounds via contraction operators}\label{sec2-1}

In this section, we study the numerical radius bounds for a single operator as well as product operators involving contraction operators. 
We begin with the following known lemma.
\begin{lemma}\cite[Lemma 1.1]{Sababheh_Moradi_Sahoo_LAMA_2024}\label{lemma12}
    Let $A,B,C\in \mathcal B\left( \mathcal H \right)$, where $A$ and $B$ are positive and let $x,y\in \mathcal H$. Then the following statements are equivalent.\\
    (i)  $\left[ \begin{matrix}
   A & {{C}^{*}}  \\
   C & B  \\
\end{matrix} \right]\in \mathcal B\left( \mathcal H\oplus \mathcal H \right)$ is positive. \\
    (ii) $|\langle Cx,y\rangle|\leq \sqrt{\langle Ax,x\rangle \langle By,y\rangle}.$\\
    (iii) There is a contraction $K$ (i.e., $\|K\|\leq 1$) such that $C=B^{1/2}KA^{1/2}.$
\end{lemma}

Using Lemma \ref{lemma12}, we now prove another equivalent statement of the positivity of the block matrix $\left[ \begin{matrix}
   A & {{C}^{*}}  \\
   C & B  \\
\end{matrix} \right]\in \mathcal B\left( \mathcal H\oplus \mathcal H \right)$.

\begin{theorem} \label{lem_3.1}  
    Let $A,B,C\in \mathcal B\left( \mathcal H \right)$, where $A$ and $B$ are positive and let $x,y\in \mathcal H$. Then
  $\left[ \begin{matrix}
   A & {{C}^{*}}  \\
   C & B  \\
\end{matrix} \right]\in \mathcal B\left( \mathcal H\oplus \mathcal H \right)$ is positive  if and only if there is a contraction $K$ such that
$\left| \left\langle Cx,y \right\rangle  \right| \le \sqrt{\left\langle {{A}^{\frac{1}{2}}}{{\left| K \right|}}{{A}^{\frac{1}{2}}}x,x \right\rangle \left\langle {{B}^{\frac{1}{2}}}{{\left| {{K}^{*}} \right|}}{{B}^{\frac{1}{2}}}y,y \right\rangle }.$ 
\end{theorem}

\begin{proof}
Suppose  $\left[ \begin{matrix}
   A & {{C}^{*}}  \\
   C & B  \\
\end{matrix} \right]\in \mathcal B\left( \mathcal H\oplus \mathcal H \right)$ is positive. Then form Lemma \ref{lemma12}, we obtain
\begin{eqnarray*}
    \left| \left\langle Cx,y \right\rangle  \right|^2=\left| \left\langle B^{1/2}KA^{1/2}x,y \right\rangle  \right|^2&=&  \left| \left\langle KA^{1/2}x,B^{1/2}y \right\rangle  \right|^2 \\
    &\leq& \langle |K|A^{1/2}x,A^{1/2}x\rangle \langle |K^*|B^{1/2}y,B^{1/2}y\rangle\\
    &=& \langle A^{1/2}|K|A^{1/2}x,x\rangle \langle B^{1/2}|K^*|B^{1/2}y,y\rangle.
\end{eqnarray*}
Conversely, let
$\left| \left\langle Cx,y \right\rangle  \right| \le \sqrt{\left\langle {{A}^{\frac{1}{2}}}{{\left| K \right|}}{{A}^{\frac{1}{2}}}x,x \right\rangle \left\langle {{B}^{\frac{1}{2}}}{{\left| {{K}^{*}} \right|}}{{B}^{\frac{1}{2}}}y,y \right\rangle }.$ Since $|K|\leq I$ and $|K^*|\leq I$, we have $ \sqrt{\left\langle {{A}^{\frac{1}{2}}}{{\left| K \right|}}{{A}^{\frac{1}{2}}}x,x \right\rangle \left\langle {{B}^{\frac{1}{2}}}{{\left| {{K}^{*}} \right|}}{{B}^{\frac{1}{2}}}y,y \right\rangle } \leq \sqrt{\left\langle Ax,x \right\rangle \left\langle By,y \right\rangle }.$ Thus, Lemma \ref{lemma12} implies $\left[ \begin{matrix}
   A & {{C}^{*}}  \\
   C & B  \\
\end{matrix} \right]\in \mathcal B\left( \mathcal H\oplus \mathcal H \right)$ is positive.
\end{proof}

Next, we need the following lemma.

\begin{lemma}   \cite{K_PRIMS_88} \label{lem2-2}
     Let $T\in \mathcal{B}(\mathcal{H})$ and $f,g$ are non-negative continuous functions on $[0,\infty)$ such that $f(t)g(t)=t$ for all $t\geq 0.$ Then
     $|\langle Tx,y \rangle|\leq \sqrt{\langle f^2(|T|)x,x\rangle \langle g^2(|T^*|)y,y\rangle}$ for all $x,y\in \mathcal{H}.$
\end{lemma}

By applying  Theorem \ref{lem_3.1} and Lemma \ref{lem2-2}, we obtain the following Kato's type inequality, which is an improved version of \cite[Corollary 2.1]{Sababheh_Moradi_Sahoo_LAMA_2024}.
\begin{proposition}\label{Cor_contraction}
Let $T\in \mathcal B\left( \mathcal H \right)$. If  $f,g$ are non-negative continuous functions on $\left[ 0,\infty  \right)$  satisfying $f\left( t \right)g\left( t \right)=t$, $t>0$, then there is a contraction operator $K$ such that
\[\left| \left\langle Tx,y \right\rangle  \right|\le \sqrt{  \left\langle f\left( \left| T \right| \right){{\left| {{K}^{}} \right|}}f\left( \left| T \right| \right)x,x \right\rangle \left\langle g\left( \left| {{T}^{*}} \right| \right){{\left| K^* \right|}}g\left( \left| {{T}^{*}} \right| \right)y,y \right\rangle  }, \quad \text{for all } x,y\in \mathcal H.\]
\end{proposition}
\begin{proof}
    From Lemma \ref{lem2-2}, we get $|\langle Tx,y\rangle|^2 \leq \langle f^2(|T|)x,x\rangle \langle g^2(|T^*|)y,y\rangle$ for all $x,y\in \mathcal{H}$ and so  from Lemma \ref{lemma12}, we get $\left[ \begin{matrix}
   f^2(|T|) & {{T}^{*}}  \\
   T & g^2(|T^*|)  \\
\end{matrix} \right]\in \mathcal B\left( \mathcal H\oplus \mathcal H \right)$ is positive. Therefore, Theorem \ref{lem_3.1} implies the desired result.
\end{proof}

Applying Proposition \ref{Cor_contraction}, we now obtain an upper bound for the numerical radius.

\begin{theorem}\label{lem_3.2} 
    Let $T\in \mathcal B\left( \mathcal H \right)$. If  $f,g$ are non-negative continuous functions on $\left[ 0,\infty  \right)$  satisfying $f(t)g(t)=t$ ($t\geq 0$), then there is a contraction operator $K$ such that
\[\begin{aligned}
   w \left( T \right)&\le \frac{1}{2}\left\| 
   f\left( \left| T \right| \right){{\left| {{K}^{}} \right|}^{}}f\left( \left| T \right| \right) + g\left( \left| {{T}^{*}} \right| \right){{\left| K^* \right|}^{}}g\left( \left| {{T}^{*}} \right| \right) \right\|  
  \le \frac{1}{2}\left\| {{f}^{2}}\left( \left| T \right| \right) + {{g}^{2}}\left( \left| {{T}^{*}} \right| \right) \right\|. 
\end{aligned}\]
In particular, for $f(t)=g(t)=\sqrt{t}$, 
\begin{eqnarray}\label{new---1}
 w \left( T \right)&\le \frac{1}{2}\left\|  \left| {{T}^{}} \right|^\frac{1}{2} {{\left| K \right|}^{}} \left| {{T}^{}} \right|^\frac{1}{2} + \left| T^* \right|^\frac{1}{2}{{\left| {{K}^{*}} \right|}^{}} \left| T^* \right|^\frac{1}{2} \right\|  \le \frac{1}{2}\big\|  \left| {{T}^{}} \right| + \left| T^* \right|  \big\|.  
\end{eqnarray}
\end{theorem}
\begin{proof}
  From Proposition \ref{Cor_contraction}, we get
\begin{eqnarray*}
    |\langle Tx,x\rangle| &\leq& \langle f(|T|)|K| f(|T|)x,x\rangle^{1/2} \langle g(|T^*|)|K^*|g(|T^*|)x,x\rangle^{1/2}\\
    &\leq& \frac12 \langle (f(|T|)|K| f(|T|) + g(|T^*|)|K^*| g(|T^*|))x,x\rangle.
\end{eqnarray*}
Taking the supremum over $\|x\|=1,$ we get the desired first bound.
\end{proof}

\begin{remark}\label{Rem_3.6}
    For any contraction operator $K$, we see that $|K|\leq |K|^{1/2}\leq I$ and $|K^*|\leq |K^*|^{1/2}\leq I$. Therefore, Theorem \ref{lem_3.2} implies the existing bound \cite[Theorem 2.5]{Sababheh_Moradi_Sahoo_LAMA_2024}, namely,
    $w \left( T \right) \le \frac{1}{2}\left\| 
   f\left( \left| T \right| \right){{\left| K \right|}^{1/2}}f\left( \left| T \right| \right) + g\left( \left| {{T}^{*}} \right| \right){{\left| K^* \right|}^{1/2}}g\left( \left| {{T}^{*}} \right| \right) \right\|$ and 
    also implies \cite[Theorem 1]{El-Haddad and F. Kittaneh}, namely,
   $w(T) \leq \frac{1}{2}\left\| |T|^{2t}+ |T^*|^{2(1-t)} \right\|, $ for all $t\in [0,1].$
\end{remark}


 Again, by applying Theorem \ref{lem_3.1} and using the positivity of $\left[ \begin{matrix}
   {{\left| B \right|}^{2}} & {{B}^{*}}A  \\
   {{A}^{*}}B & {{\left| A \right|}^{2}}  \\
\end{matrix} \right]\in \mathcal{B}(\mathcal{H}\oplus \mathcal{H})$, we can also obtain the following numerical radius bound for the product two operators, which improves the bound in  \cite[Remark 2.4]{Sababheh_Moradi_Sahoo_LAMA_2024} as well as \cite[(17)]{Kittaneh_STD_2005}.
\begin{corollary}\label{cor---1}
    If $A,B\in \mathcal B\left( \mathcal H \right)$, then there is a contraction operator $K$ such that
$$\begin{aligned}
   w \left( {{A}^{*}}B \right)&\le \frac{1}{2}\Big\|  \left| A \right|{{\left| {{K}^{*}} \right|}}\left| A \right|+ \left| B \right|{{\left| K \right|}}\left| B \right| \Big\| 
   \le \frac{1}{2}\left\| A^*A+B^*B \right\|.  
\end{aligned}$$
\end{corollary}

 \noindent 
 
 The following example illustrates that the inequalities in Corollary \ref{cor---1} are proper.
 Consider  $A=\begin{bmatrix}
     0 & 1& 0\\
     0 &0 &2\\
     0&0&0
 \end{bmatrix}$ and $B=I.$
 Then there exists a contraction $K= \begin{bmatrix}
     0 & 0& 0\\
     1 &0 &0\\
     0&1&0
 \end{bmatrix}$ such that $A^*B=|A|K|B|$, and  we see that
$$ w(A^*B)=\frac{\sqrt{5}}{2} <\frac{1}{2}\Big\|  \left| A \right|{{\left| {{K}^{*}} \right|}}\left| A \right|+ \left| B \right|{{\left| K \right|}}\left| B \right| \Big\| =2< \frac{5}{2}
   = \frac{1}{2}\left\| A^*A+B^*B \right\|. $$

\section{Schatten $p$-norm and $p$-numerical radius bounds via singular values}\label{sec3}

\noindent

In this section, we develop the Schatten $p$-norm inequalities for the sum of two $n\times n$ complex matrices and from which we deduce the $p$-numerical radius bounds. To get the results first we need the following lemma.


\begin{lemma}\label{lem2-3} \cite[Corollary III.1.2]{Bhatia}
    Let $T\in \mathcal{M}_n(\mathbb C).$ Then $$s_j(T)= \max_{\dim \, M=j} \left\{ \min_{x\in M,\, \|x\|=1} \|Tx\|\right\}. $$
    
\end{lemma}


\begin{theorem}\label{th2-1}
    Let $T,S \in \mathcal{M}_n(\mathbb{C})$ and $e,f,g,h$ are non-negative continuous functions on $[0,\infty)$ such that $f(t)g(t)=t$ and $e(t)h(t)=t$ for all $t\geq 0.$ Then 
    \begin{eqnarray*}\label{}
        s_j(T+S) \leq \left\| g^2(|T^*|)+h^2(|S^*|) \right\|^{1/2} \, s_j^{1/2} \left(f^2(|T|)+e^2(|S|)  \right) \quad \text{for every $j=1,2, \ldots,n$}.
    \end{eqnarray*}
  Furthermore,
    \begin{eqnarray}\label{eq1}
        \|T+S\|_p \leq \left\| g^2(|T^*|)+h^2(|S^*|) \right\|^{1/2} \,  \left \|f^2(|T|)+e^2(|S|)  \right\|_{p/2}^{1/2} \quad \text{for every $p>0$ }.
    \end{eqnarray}
\end{theorem}

\begin{proof}
    From Lemma \ref{lemma12} and Lemma \ref{lem2-2}, we get that 
    $\begin{pmatrix}
        f^2(|T|) & T^*\\
        T& g^2(|T^*|)
    \end{pmatrix} \geq 0$ and $
    \begin{pmatrix}
        e^2(|S|) & S ^*\\
        S& h^2(|S^*|)
    \end{pmatrix} \geq 0.$ Therefore, $\begin{pmatrix}
        f^2(|T|)+e^2(|S|) & T^*+S^*\\
        T+S& g^2(|T^*|)+h^2(|S^*|)
    \end{pmatrix} \geq 0$
    and so again from Lemma \ref{lemma12}, we get
    $$|\langle (T+S)x,y \rangle|^2\leq \langle (f^2(|T|)+e^2(|S|))x,x\rangle \langle (g^2(|T^*|)+h^2(|S^*|)) y,y\rangle, \quad \text{for all } x,y\in \mathcal{H}.$$ Taking the supremum over $\|y\|=1$, we get
    $$\|(T+S)x\|^2\leq  \| g^2(|T^*|)+h^2(|S^*|)\| \langle (f^2(|T|)+e^2(|S|))x,x\rangle, \quad \text{for all $x\in \mathcal{H}.$} $$  This implies 
    $$ \max_{\dim \, M=j}  \min_{x\in M,\, \|x\|=1} \|(T+S)x\|^2\leq  \| g^2(|T^*|)+h^2(|S^*|)\| 
    \max_{\dim \, M=j}  \min_{x\in M,\, \|x\|=1} \langle (f^2(|T|)+e^2(|S|))x,x\rangle, $$ for every $j=1,2, \ldots,n.$
   Therefore, using Lemma \ref{lem2-3}, we get
    $$ s_j(T+S) \leq  \| g^2(|T^*|)+h^2(|S^*|)\|^{1/2} s_j^{1/2}(f^2(|T|)+e^2(|S|))\quad \text{for every $j=1,2, \ldots,n.$ } $$
    Now we have
    $$ \left( \sum_{j=1}^ns_j^p(T+S)\right)^{1/p} \leq  \| g^2(|T^*|)+h^2(|S^*|)\|^{1/2} \left( \sum_{j=1}^n  s_j^{p/2}(f^2(|T|)+e^2(|S|)) \right)^{1/p}.$$
    This gives \eqref{eq1}.
\end{proof}

From the inequality \eqref{eq1}, we deduce that

\begin{corollary} \label{cor2-1-1}
    Let $T,S \in \mathcal{M}_n(\mathbb{C})$. Then  
\begin{eqnarray}\label{eq1-1}
        \|T+S\|_p \leq \left\| |T^*|^{2t}+|S^*|^{2t} \right\|^{1/2}   \left \||T|^{2(1-t)}+|S|^{2(1-t)}  \right\|_{p/2}^{1/2} 
    \end{eqnarray}
    and 
    \begin{eqnarray}\label{eq2-1}
        \|T+S\|_p \leq \left\| |T^*|^{2t}+|S^*|^{2(1-t)} \right\|^{1/2}   \left \||T|^{2(1-t)}+|S|^{2t}  \right\|_{p/2}^{1/2},
    \end{eqnarray}
    for every $p>0$ and for all $t\in [0,1].$ 
    \end{corollary}

    Considering $p\to \infty$ in Corollary \ref{cor2-1-1}, we get
    
    \begin{corollary} \label{cor3-1-1}
    Let $T,S \in \mathcal{M}_n(\mathbb{C})$. Then 
    \begin{eqnarray}\label{eq1-1-1}
        \|T+S\| \leq \left\| |T^*|^{2t}+|S^*|^{2t} \right\|^{1/2}   \left \||T|^{2(1-t)}+|S|^{2(1-t)}  \right\|_{}^{1/2} 
    \end{eqnarray}
    and 
    \begin{eqnarray}\label{eq2-1-1}
        \|T+S\| \leq \left\| |T^*|^{2t}+|S^*|^{2(1-t)} \right\|^{1/2}   \left \||T|^{2(1-t)}+|S|^{2t}  \right\|_{}^{1/2},
    \end{eqnarray}
     for all $t\in [0,1].$
\end{corollary}

\begin{remark} Let $T,S \in \mathcal{M}_n(\mathbb{C})$.
     In particular, for $t=1/2$, we get
     \begin{eqnarray}\label{p222}
        \|T+S\| \leq  \left\| |T^*|^{}+|S^*|^{} \right\|^{1/2}   \left \||T|^{}+|S|^{}  \right\|_{}^{1/2}.
     \end{eqnarray}
Moreover, if $T$ and $S$ are normal, then 
\begin{eqnarray}\label{p222-n}
        \|T+S\| \leq  \big \||T|^{}+|S|^{}  \big\|.
     \end{eqnarray}
\end{remark}

Using Theorem \ref{th2-1}, we now obtain the $p$-numerical radius bound.

\begin{theorem}\label{cor2-1}
    Let $T \in \mathcal{M}_n(\mathbb{C})$ and let $e,f,g,h$ are non-negative continuous functions on $[0,\infty)$ such that $f(t)g(t)=t$ and $e(t)h(t)=t$ for all $t\geq 0.$ Then
     \begin{eqnarray*}
        w_p(T) &\leq & \frac12 \left\| g^2(|T^*|)+h^2(|T|) \right\|^{1/2} \,  \left \|f^2(|T|)+e^2(|T^*|)  \right\|_{p/2}^{1/2}, \quad \text{for every $p>0$ }.
    \end{eqnarray*}
\end{theorem}

\begin{proof}
    Replacing $T$ by $e^{i\theta}T$ and $S$ by $e^{-i\theta}S^*$ ($\theta \in \mathbb R$) in Theorem \ref{th2-1}, we get
    \begin{eqnarray*}
        \| Re( e^{i\theta}T)\|_p &\leq & \frac12 \left\| g^2(|T^*|)+h^2(|T|) \right\|^{1/2} \,  \left \|f^2(|T|)+e^2(|T^*|)  \right\|_{p/2}^{1/2}, \quad \text{for every $p>0$ }.
    \end{eqnarray*}
   Considering the supremum over $\theta \in \mathbb R$, we get the desired inequality.
\end{proof}

From Theorem \ref{cor2-1}, we deduce that
\begin{corollary}\label{cor2-2}
Let $T \in \mathcal{M}_n(\mathbb{C})$, then
    \begin{eqnarray*}\label{eq2}
        w_p(T) &\leq& \frac12\left\| |T|^{2(1-t)}+|T^*|^{2(1-t)} \right\|^{1/2} \,  \left \||T|^{2t}+|T^*|^{2t}  \right\|_{p/2}^{1/2},
      \end{eqnarray*}  
  \text{for every $p>0$ and  for all $t\in [0,1]$ }.
Considering $p\to \infty$, we get
    \begin{eqnarray}\label{cor2-2-1}
        w^2(T) &\leq& \frac14\left\| |T|^{2(1-t)}+|T^*|^{2(1-t)} \right\|^{} \,  \left \||T|^{2t}+|T^*|^{2t}  \right\|_{}^{},\quad  \text{ for all $t\in [0,1]$ }.
      \end{eqnarray}  

\end{corollary}

\begin{remark}
For $T \in \mathcal{M}_n(\mathbb{C})$, we see that
    \begin{eqnarray*}
         \frac14\left\| |T|^{2(1-t)}+|T^*|^{2(1-t)} \right\|^{} \,  \left \||T|^{2t}+|T^*|^{2t}  \right\|_{}^{} &=& \left\| \frac{|T|^{2(1-t)}+|T^*|^{2(1-t)}}{2} \right\|^{}   \left \|\frac{|T|^{2t}+|T^*|^{2t} }{2} \right\|_{}^{} \\
         &\leq& \left\| \frac{|T|^{2}+|T^*|^{2}}{2} \right\|^{1-t}   \left \|\frac{|T|^{2}+|T^*|^{2} }{2} \right\|_{}^{t} \\
&=& \frac12 \left\| |T|^2+ |T^*|^2\right\| \quad \textit{for all $t\in [0,1]$}.
\end{eqnarray*}
Therefore, the bound \eqref{cor2-2-1} refines the well known bound $w^2(T) \leq \frac12 \left\| |T|^2+ |T^*|^2\right\|,$ given in \cite{Kittaneh_STD_2005}. Also, we remark that the bound \eqref{k1} follows from \eqref{cor2-2-1} by setting $t=1/2.$
\end{remark}

From Theorem \ref{cor2-1}, we also deduce that
\begin{corollary}\label{cor2-2-10}
Let $T \in \mathcal{M}_n(\mathbb{C})$, then
    \begin{eqnarray*}\label{eq2.}
        w_p(T) &\leq& \frac12\left\| |T|^{2t}+|T^*|^{2(1-t)} \right\|^{1/2} \,  \left \||T|^{2t}+|T^*|^{2(1-t)}  \right\|_{p/2}^{1/2},
      \end{eqnarray*}  
  \text{for every $p>0$ and  for all $t\in [0,1]$ }.

\end{corollary}

Considering $p\to \infty$, we get
    $w(T) \leq \frac12   \left \||T|^{2t}+|T^*|^{2(1-t)}  \right\|_{}^{},$ $ \text{ for all $t\in [0,1]$ },$ which is also proved in \cite[Theorem 1]{El-Haddad and F. Kittaneh}.
Now from Corollary \ref{cor2-2} (for $t=1/2$), we get

\begin{corollary}\label{cor2-3}
Let $T \in \mathcal{M}_n(\mathbb{C})$, then
    \begin{eqnarray*}\label{eq3}
        w_p(T) &\leq& \frac12\left\| |T|+|T^*| \right\|^{1/2} \,  \left \||T|+|T^*|  \right\|_{p/2}^{1/2}, \quad \text{\text{for every $p>0$. }}
      \end{eqnarray*}  
   
\end{corollary}
Taking $p\to \infty$, we get the existing bound \eqref{k1}.
Considering $T=S$ in \eqref{eq1-1}, we get
\begin{corollary}\label{p-norm}
    Let $T \in \mathcal{M}_n(\mathbb{C})$, then
    \begin{eqnarray}
        \|T\|_p &\leq& \|T\|^{t} \|T\|^{1-t}_{p(1-t)}, \quad \text{for every $p>0$ and for all $t\in (0,1)$.}
    \end{eqnarray}
    In particular, for $t=1/2$, 
    \begin{eqnarray}\label{appl}
        \|T\|_p  &\leq& \sqrt{ \|T\|^{} \|T\|^{}_{p/2} },\quad \text{for every $p>0$.}
    \end{eqnarray}
\end{corollary}

\section{An application to estimate the energy of a graph}\label{sec4}

\noindent 

As an application of the Schatten $p$-norm inequalities obtained above, we develop a lower bound for the energy of a simple graph, which was introduced by Gutman \cite{Gutman} in connection to the
total $\pi$-electron energy. For details on the general theory of the total $\pi$-electron energy, as well as its chemical applications, the reader can see \cite{Appl1, Appl2}. Let $G$ be a simple (undirected) graph with $n$ vertices $v_1,v_2,\ldots, v_n$ and $m$ edges $e_1,e_2,\ldots,e_m.$ 
Let $d_i$ denote the degree of the vertex $v_i$ for $i=1,2,\ldots,n.$
The adjacency matrix associated with the graph $G$, denoted as $\textit{Adj}(G)$, is defined as $\textit{Adj}(G)=[a_{ij}]_{n\times n}$, where $a_{ij} = 1$ if $ v_i \sim v_j$ (i.e., $v_i$ is adjacent to $v_j$)  and $a_{ij} = 0$ otherwise. Clearly, $\textit{Adj}(G)$ is a self-adjoint matrix with entries $0,1$ and the main diagonal entries are zero. Let $\lambda_1(G)\geq  \lambda_2(G)\geq \ldots \geq \lambda_n(G)$ be the eigenvalues of $\textit{Adj}(G)$. It is well known that $\lambda_1(G)> 0$ and so $\| \textit{Adj}(G) \|= \lambda_1(G).$ In 2019, Bhunia et al. \cite[Theorem 2.2]{Bhunia_E} proved that
\begin{eqnarray}\label{eig}
    \| \textit{Adj}(G) \| &\leq & \sqrt{ \max_{1\leq i \leq n} \left\{ \sum_{j, v_i \sim v_j}d_j\right\} }.
\end{eqnarray}
The energy of the graph $G$, denoted as $\mathcal{E}(G)$, is defined as $\mathcal{E}(G)=\sum_{i=1}^n|\lambda_i(G)|.$
The search of lower and upper bounds for $\mathcal{E}(G)$ is a wide subfield of the spectral graph theory. In 1971,  McClelland \cite{Energy} provided an upper bound that 
   $ \mathcal{E}(G) \leq \sqrt{2mn}.$
After that various bounds have been studied, we refer to see \cite{Jahan, Rada1} and the references therein. Recently, Bhunia \cite{Bhunia_G} showed that 
    $ \mathcal{E}(G) \leq  \sqrt{2m \left(\textit{rank Adj}(G) \right)}.$
Clearly, $\sqrt{2m \left(\textit{rank Adj}(G) \right)} <\sqrt{2mn} $ for every singular graph $G$. A lower bound is provided in \cite{C_lower} that
    $ \mathcal{E}(G) \geq 2\sqrt{m}.  $
We now develop a new lower bound of $\mathcal{E}(G)$ in terms of the degree of the vertices and the number of edges.

\begin{theorem}\label{Graph-lower bound}
Let $G$ be a simple graph with $m$ edges and $n$ vertices $v_1,v_2,\ldots,v_n$ such that degree of $v_i$ is $d_i$ for each $i=1,2,\ldots,n.$ Then
 \begin{eqnarray*}
     \mathcal{E}(G) &\geq&  \frac{2m}{   \sqrt{ \max_{1\leq i \leq n} \left\{ \sum_{j, v_i \sim v_j}d_j\right\}}  }.  
\end{eqnarray*}
\end{theorem}

\begin{proof}
    From \eqref{appl}, we obtain that
    $\| \textit{Adj}(G) \|_p^2\leq \| \textit{Adj}(G) \| \| \textit{Adj}(G) \|_{p/2},$  \text{for every $p>0$}. In particular, for $p=2$, we get
    \begin{eqnarray*}
      2m= \sum_{i=1}^n d_i=  \text{trace} \, |\textit{Adj}(G)|^2 &=& \| \textit{Adj}(G) \|_2^2\\
      &\leq& \| \textit{Adj}(G) \| \| \textit{Adj}(G) \|_{1}\\
      &=& \| \textit{Adj}(G) \| \sum_{i=1}^n | \lambda_i(G)|\\
      &\leq& \sum_{i=1}^n | \lambda_i(G)|  \sqrt{ \max_{1\leq i \leq n} \left\{ \sum_{j, v_i \sim v_j}d_j\right\}}  \quad (\text{by \eqref{eig}}).
    \end{eqnarray*}
    Therefore,    
     $\mathcal{E}(G)=\sum_{i=1}^n | \lambda_i(G)| \geq \frac{2m}{   \sqrt{ \max_{1\leq i \leq n} \left\{ \sum_{j, v_i \sim v_j}d_j\right\}}  },  $ as desired.
\end{proof}

\begin{remark}
Clearly, Theorem \ref{Graph-lower bound} gives better bound than the bound  $ \mathcal{E}(G) \geq 2\sqrt{m}  $ if $$m>  \sqrt{ \max_{1\leq i \leq n} \left\{ \sum_{j, v_i \sim v_j}d_j\right\}}.$$ 
We consider two graphs $G_1$ and $G_2$ and their adjacency matrices (known as Huckel matrices \cite{JMNT}) associated with the molecular structure of the carbon skeleton of $1,2$-divinylcyclobutadiene and $1,4$-divinylbenzene, respectively.
 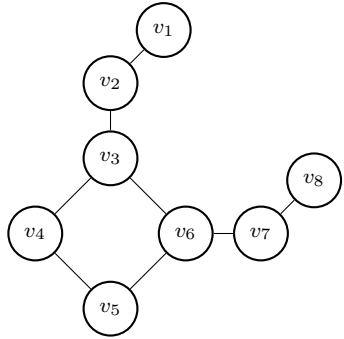
\begin{figure}[!h]
 \begin{tikzpicture}
        [auto=left,every node/.style={circle, draw=black, thick, minimum size=0.0cm}]
        \node (v4) at (0,0) {\tiny{$v_4$}};
        \node (v3) at (1,1) {\tiny{$v_3$}};
        \node (v6) at (2,0) {\tiny{$v_6$}};
        \node (v5) at (1,-1){\tiny{$v_5$}};
				\node (v2) at (1,2) {\tiny{$v_2$}};
				\node (v1) at (1.707,2.707) {\tiny{$v_1$}};
				\node (v7) at (3,0) {\tiny{$v_7$}};
				\node (v8) at (3.707,0.707) {\tiny{$v_8$}};
        \draw[-] (v1) to (v2);
        \draw[-] (v2) to (v3);
        \draw[-] (v3) to (v4);
				\draw[-] (v3) to (v6);
				\draw[-] (v4) to (v5);
			  \draw[-] (v5) to (v6); 
				\draw[-] (v6) to (v7);
				\draw[-] (v7) to (v8);


	
	\end{tikzpicture}
	 \caption{Graph $G_1$ } 
	\end{figure}

		\begin{figure}[!h]
	\begin{tikzpicture}
        [auto=left,every node/.style={circle, draw=black, thick, minimum size=0.0cm}]
        \node (v6) at (0,0) {\tiny {$v_6$}};
        \node (v7) at (.707,.707) {\tiny{$v_7$}};
        \node (v8) at (.707,1.707) {\tiny{$v_8$}};
        \node (v3) at (0,2.414) {\tiny{$v_3$}};
				\node (v4) at (-.707,1.7072) {\tiny{$v_4$}};
				\node (v5) at (-.707,.707) {\tiny{$v_5$}};
				\node (v2) at (0,3.414) {\tiny{$v_2$}};
				\node (v1) at (0,4.414) {\tiny{$v_1$}};
				\node (v9) at (0,-1) {\tiny{$v_9$}};
				\node (v10) at (0,-2) {\tiny{$v_{10}$}};
				
        \draw[-] (v1) to (v2);
        \draw[-] (v2) to (v3);
        \draw[-] (v3) to (v4);
        \draw[-] (v4) to (v5);
				\draw[-] (v5) to (v6);
			  \draw[-] (v6) to (v7); 
				\draw[-] (v8) to (v7);
				\draw[-] (v3) to (v8);
				\draw[-] (v6) to (v9);
				\draw[-] (v9) to (v10);
		
\end{tikzpicture}
\caption{Graph $G_2$ } 
\end{figure}
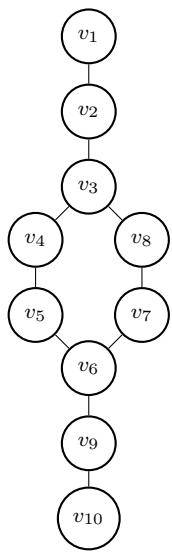

From Theorem \ref{Graph-lower bound} we get $\mathcal{E}(G_1)\geq \frac{16}{\sqrt{7}}$ and $\mathcal{E}(G_2)\geq \frac{20}{\sqrt{6}}$, whereas the existing bound $ \mathcal{E}(G) \geq 2\sqrt{m}  $  gives $\mathcal{E}(G_1)\geq 2\sqrt{8}$ and $\mathcal{E}(G_2)\geq 2\sqrt{10}.$
\end{remark}

\bigskip

\noindent \textbf{Data availability statements}
Data sharing not applicable to this article as no datasets were generated or analysed during the current study.

\noindent \textbf{Competing Interests}
The authors have no relevant financial or non-financial interests to disclose.


\end{document}